\documentclass{article}
\usepackage{amsmath,amssymb,amsthm,mathtools,graphicx,epsfig,float,url}
\usepackage{mathrsfs}  
\usepackage{dsfont}
\usepackage{esint}
\usepackage{braket}
\usepackage[dvipsnames]{xcolor}
\usepackage[colorlinks=true,citecolor={Plum},linkcolor={Periwinkle}]{hyperref}
\usepackage{cleveref}
\usepackage[a4paper,
            left=2cm,
            right=2cm,
            top=2cm,
            bottom=2cm]{geometry}

\usepackage{biblatex}
\newcommand{\R}{\textnormal{I\kern-0.21emR}}
\newcommand{\N}{\textnormal{I\kern-0.21emN}}

\def\Om{{\Omega}}
\allowdisplaybreaks

\newtheorem{material}{material}
\newtheorem{theorem}[material]{Theorem}
\newtheorem{proposition}[material]{Proposition}

\newtheorem{corollary}[material]{Corollary}
\newtheorem{definition}[material]{Definition}
\newtheorem{lemma}[material]{Lemma}

\newtheorem{remark}[material]{Remark}

\def\B{{\mathbb B}}
\def\O{{\Omega}}
\def\n{{\nabla}}

\def\sh{{\sharp}}

\def\e{{\varepsilon}}
\usepackage{xpatch}
\makeatletter
\newcounter{proofpart}
\xpretocmd{\proof}{\setcounter{proofpart}{0}}{}{}
\newcommand{\proofpart}[1]{%
  \par 
  \addvspace{\medskipamount}%
  \stepcounter{proofpart}%
  \noindent\emph{Step \theproofpart: #1}\par\nobreak\smallskip
  \@afterheading
}
\makeatother %

\newcommand{\al}{\alpha}
\newcommand{\be}{\beta}

\newcommand{\na}{\nabla}

\newcommand{\si}{\sigma}

\DeclarePairedDelimiter{\abs}{\lvert}{\rvert}

\numberwithin{equation}{section}

\begin{document}
\title{Sharp stability for the Talenti inequality}

\author{Paolo Acampora\footnote{Dipartimento di Bioscienze e Territorio, Università degli Studi del Molise, Viale dell’università 3,
86090, Pesche (IS) (\texttt{paolo.acampora@unimol.it})}, Jimmy Lamboley\footnote{École Normale Supérieure, CNRS, PSL University, DMA and Sorbonne Université, Université Paris Cité, IMJ-
PRG, F-75005 Paris, France (\texttt{jimmy.lamboley@ens.fr})}, Idriss Mazari-Fouquer\footnote{CEREMADE, UMR CNRS 7534, Universit\'e Paris-Dauphine, Universit\'e PSL, Place du Mar\'echal De Lattre De Tassigny, 75775 Paris
cedex 16, France (\texttt{mazari@ceremade.dauphine.fr})}}
\def\P{{\Phi}}

\maketitle
\begin{abstract}
We prove a sharp quantitative version of the Talenti rearrangement inequality. Specifically, we show that, for any $p\in [1;+\infty]$, for any $f\in L^\infty(\O),\, f\geq 0,\, f\not\equiv 0$, for any  open set $\O\subset \R^d$ with finite volume, the solution $u_{f,\O}$ to $-\Delta u_{f,\O}=f,\, u_{f,\O} \in W^{1,2}_0(\O)$, satisfies, in dimension $d\geq 1$,  $\Vert u_{f^\sharp,\O^\sharp}\Vert_{L^p(\O^\sharp)}-\Vert u_{f,\O}\Vert_{L^p(\O)}\gtrsim \alpha(\O)^2+\inf_{z\in \R^d}\mathscr A(f,z)^2$ and, in dimension 1, $ \Vert u_{f^\sharp,\O^\sharp}\Vert_{L^p(\O^\sharp)}-\Vert u_{f,\O}\Vert_{L^p(\O)}\gtrsim \alpha(\O)+\inf_{z\in \R}\mathscr A(f,z)^2$ where $\alpha(\O)$ denotes the Fraenkel asymmetry of $\O$,  $\mathscr A(f,z)$ a weighted asymmetry of $f$ with respect to some center $z$ and $\O^\sharp$ (resp. $f^\sharp$) is the Schwarz rearrangement of $\O$ (resp. of $f$). In dimension $d\geq 2$, this estimate implies $ \Vert u_{f^\sharp,\O^\sharp}\Vert_{L^p(\O^\sharp)}-\Vert u_{f,\O}\Vert_{L^p(\O)}\gtrsim \alpha(\O)^2+\inf_{z\in \R^d}\Vert f-f^\sh(\cdot-z)\Vert_{L^1}^2.$ Our result sharpens and extends \cite{zbMATH08185219,arXiv:2311.18617} and, unlike  \cite{zbMATH08185219}, does not rely on shape differentiation arguments but on a generalisation of the quantitative Saint-Venant inequality.
\end{abstract}


\textbf{Acknowledgment: } J. L. and I. M-F. are partially supported by the ANR STOIQUES project. Parts of this work were carried out while I. M-F. was supported by a PSL Young Researcher Starting Grant. 
The author P. A. was partially supported by the INdAM - GNAMPA Project, 2026,
”Equazioni alle Derivate Parziali: problemi inversi e stime quantitative”, CUP\_53C25002010001.

\textbf{LLM Statement: } Ideas and proofs were developed without the use of LLMs. ChatGPT 6 was used in the last stages of writing to proofread the manuscript.

\tableofcontents
\section{Introduction}
\subsection{Scope of the paper and state of the art}
We prove a sharp version of the Talenti inequality.  Although we refer to \cite{zbMATH07035366} for a general reference, let us recall basic facts about the classical Talenti inequality: for a bounded open set $\O\subset \R^d$, let $\O^\sharp$ denote its Schwarz rearrangement, that is, the centered ball with the same volume as $\O$. For any $v\in L^2(\O)$, $v\geq 0$, let $v^\sharp$ denote its Schwarz rearrangement, that is, the unique radially symmetric, non-increasing function with the same distribution function as $v$. Finally,  for any $f\in L^\infty(\O)$, let $u_{f,\O}$ be the unique weak solution to 
\begin{equation}\label{Eq:Poisson}\begin{cases}-\Delta u_{f,\O}=f&\text{ in }\O,
\\ u_{f,\O}\in W^{1,2}_0(\O).\end{cases}\end{equation}
The Talenti inequality \cite[Theorem 1]{zbMATH03531830} asserts that if $f\in L^\infty(\O),\, f\geq 0$
\[ 0\leq u_{f,\O}^\sh\leq u_{f^\sh,\O^\sh}.\] In fact, the proof of \cite[Theorem 1]{zbMATH03531830} shows the following stronger result: the map
\[ v_{f,\O}:=u_{f^\sh,\O^\sh}-u_{f,\O}^\sh\] is radially monotone non-increasing. In particular, this entails  that for any $p\in [1;+\infty]$
\begin{equation}\label{Eq:LpTalenti}
\left\Vert u_{f,\O}\right\Vert_{L^p(\O)} \leq \left\Vert u_{f^\sharp,\O^\sharp}\right\Vert_{L^p(\O^\sharp)}.\end{equation} Furthermore, this estimate is rigid \cite[Theorem 1]{zbMATH03962241} in the sense that, if $f\not\equiv 0$,   equality holds in \eqref{Eq:LpTalenti} for some $p\in [1;+\infty]$ if, and only if, up to a translation, $\O=\O^\sh$ (up to zero capacity sets) and $f=f^\sh$ (up to null sets). 

The central object of the present article is the quantified version of this rigidity estimate. Observe that two objects are symmetrised in \eqref{Eq:LpTalenti}, the domain $\O$ and the source term $f$, so that stability must be established with respect to both. Although we defer a discussion of the literature to Section \ref{Se:Bib}, let us mention here that non-sharp stable versions of the quantitative Talenti inequality were previously established \cite{arXiv:2311.18617,K21}. When the domain $\O$ is assumed to be a ball and $f$ is restricted to being a characteristic function, a sharp quantitative version of \eqref{Eq:LpTalenti} was proved by two of the authors in the recent paper \cite{zbMATH08185219}. In the present paper, we introduce the sharp asymmetry required to handle any bounded source term, and we provide a sharp version of \eqref{Eq:LpTalenti}, sharpness being understood with respect to both $\O$ and $f$. Notably, our proof only relies on a strengthened version of the quantitative Saint-Venant inequality from Brasco, De Philippis \& Velichkov \cite{zbMATH06464861}.

\subsection{Main results}
Before we state our main results, we recall some basic notions about Schwarz rearrangement:
\begin{definition}\label{De:Schwarz}
  Let $\Omega\subset\R^d$ be a measurable set of finite measure, and $f: \Omega \to \R$ be a measurable, non-negative function. We define:
           \begin{itemize}
           \item the \emph{distribution function} $\mu_f : [0,+\infty)\, \to [0, +\infty)$ of $f$ as the function
	       $$
	           \forall t\geq 0, \qquad\mu_f(t)= \abs{\Set{x \in \Omega \, ,\,  {f(x)} >  t}}.
	       $$
           \item the \emph{decreasing rearrangement} $f^*:[0,+\infty)\to[0,+\infty]$ of $f$ as 
         \[
           \forall s\geq 0, \qquad f^*(s) = \inf\Set{t>0 \, ,\; \mu_f(t) \leq s}.
         \] 
        \item the \emph{Schwarz rearrangement} $f^\sharp:\Om^\sharp\to[0,+\infty]$ of $f$ as
            \[\forall x\in\Om^\sharp, \qquad
            f^\sharp (x)= f^*(\omega_d \abs{x}^d),
            \]
            where $\Om^\sh$ denotes the centered ball having the same volume as $\Om$ and $\omega_d=|\B(0;1)|.$
            \end{itemize} Alternatively, one can characterise the Schwarz rearrangement $f^\sharp$ as the unique (up to null sets) radially symmetric, non-increasing function $f^\sharp:\O^\sharp\to [0;+\infty]$ satisfying 
\[ \mu_f=\mu_{f^\sharp}.\]
\end{definition} 

We now give a proper definition of the two types of asymmetries we consider. 
\begin{definition}\label{De:Asymmetry}
\begin{enumerate}
\item\underline{Domain asymmetry:} we use the classical Fraenkel asymmetry defined as
\[
\alpha(\Omega)
:=
\inf\left\{
\frac{|\Omega\triangle \B(x_0;r)|}{|\Omega|},\;\;x_0\in\mathbb R^d, \;\;|\B(x_0;r)|=|\Omega|\right\}.
\]
\item \underline{Function asymmetry:} for any $f\in L^\infty(\O),\, f\geq 0$ and any $z \in \R^d$, the weighted  asymmetry of $f$ with respect to $z$ is defined as 
\[\mathscr A(f,z)^2:=\int_0^\infty \mu_f(t)^{\frac2d-1} \left|\{f> t\}\triangle \{f^\sharp(\cdot-z)> t\}\right|^2dt\] where, by convention, $f$ (resp. $f^\sharp$)  is extended by zero outside of $\O$ (resp. $\O^\sharp$).
\end{enumerate}
\end{definition}
We will discuss why $\mathscr A$ is a natural asymmetry for $f$ in Remark \ref{ReALL}.

\begin{remark}
    One should observe that $\mathscr A$ is well-defined with the following convention: as
    \[ \left|\{f> t\}\triangle \{f^\sharp(\cdot-z)> t\}\right|\leq 2\mu_{f}(t),\] we take the integrand to be zero when $\mu_f(t)=0$. Also observe that 
    \[ \mathscr A(f,z)^2\leq 4 \int_0^\infty \mu_f(t)^{1+\frac2d}dt<\infty.\] Therefore, if $f$ is bounded, $\mu_f$ is compactly supported and so $\mathscr A(f,z)$ is  a finite quantity.
\end{remark}

In this paper, we consider a function $j\in \mathscr C^{1}((0;+\infty))\cap \mathscr C^0([0;+\infty))$ satisfying 
\begin{multline}\label{Eq:Hypj}
j\text{ is convex, }j'>0\text{ on }(0;+\infty) \text{ and }\\\text{ for any $K>0$, there exists $c>0$ such that } \forall 0< x\leq K,\, j(x)-j(0)\geq c xj'(x). \end{multline} 
Note that for any $p\geq 1$ the map $j:x\mapsto x^p$ satisfies these assumptions. Our main result is the following:
\begin{theorem}\label{Th:Main} 
Let $j$ satisfy \eqref{Eq:Hypj}. Let $\O$ be an open set with finite volume, and let $f\in L^\infty(\O)$, $f\geq 0$, $f\not\equiv 0$. 
\begin{enumerate}
    \item In dimension $d\geq 2$, there exists a constant $c>0$ depending only on $j$, on the dimension, on $\Vert f^\sharp\Vert_{L^\infty(\O^\sharp)},\, \Vert f^\sharp\Vert_{L^1(\O^\sharp)}$ and on $|\O^\sharp|$ such that the following holds: 
\begin{equation}\label{Eq:Talentij}\int_{\O^\sh}j(u_{f^\sharp,\O^\sharp})-\int_\O j(u_{f,\O})\geq c\left(\alpha(\O)^2+\inf_{z\in \R^d}\mathscr A (f,z)^2\right).\end{equation}
    \item In dimension $d=1$, there exists a constant $c>0$ depending only on $j$, on $\Vert f^\sharp\Vert_{L^\infty(\O^\sharp)},\, \Vert f^\sharp\Vert_{L^1(\O^\sharp)}$ and on $|\O^\sharp|$ such that the following holds: 
\begin{equation}\label{Eq:Talentij1d}\int_{\O^\sh}j(u_{f^\sharp,\O^\sharp})-\int_\O j(u_{f,\O})\geq c\left(\alpha(\O)+\inf_{z\in \R^d}\mathscr A (f,z)^2\right).\end{equation}
\end{enumerate}
\end{theorem}
\begin{remark}\label{Re:SV1d}
    Regarding the different exponent (for the Fraenkel asymmetry of the domain) in dimension 1, an important part of the proof is showing that 
    \[ \int_{\O^\sh}j(u_{f^\sharp,\O^\sharp})-\int_\O j(u_{f,\O})\geq c \left(T(\O^\sharp)-T(\O)\right).\] We refer to Proposition \ref{Pr:FA}. We then use the quantitative Saint-Venant inequality, for which the stability exponent depends on the dimension: while, in dimension $d\geq 2$, the optimal exponent is 2, it is 1 in dimension 1. Since we could not locate the one-dimensional result in the literature, we prove it in Lemma \ref{Le:SV1d}.
\end{remark}
\begin{remark}\label{Re:AL}
In the proof of Theorem \ref{Th:Main} we actually prove the following stronger statement in dimension $d\geq 2$: for any $R>0$, for any $\O\subset \B(0;R)$, there exists a constant $c=c(R,d,j,\Vert  f^\sharp\Vert_{L^\infty(\O^\sharp)},\, \Vert f^\sharp\Vert_{L^1(\O^\sharp)})$ such that
\[\int_{\O^\sh}j(u_{f^\sharp,\O^\sharp})-\int_\O j(u_{f,\O})\geq c\left(\alpha(\O)^2+\mathscr A(f,x_\O)^2\right)\] where $x_\O=\frac1{|\O|}\int_\O x dx$ is the barycenter of $\O$. In dimension 1, we obtain the same inequality with $\alpha(\O)^2$ replaced with $\alpha(\O)$. Removing the constraint $\O\subset \B(0;R)$ leads to having to consider other possible centers $z$. For example, in the case $\O=\B(0;1)=\O^\sharp$, our proof gives 
\[ \int_{\O^\sharp} j(u_{f^\sharp,\O^\sharp})-\int_{\O^\sharp} j(u_{f,\O^\sharp})\geq c \int_0^\infty \mu_f(t)^{\frac2d-1}|\{f> t\}\triangle \{f^\sharp > t\}|^2dt.\] In Section \ref{Se:Bib} we explain how this allows us to recover the results of \cite{zbMATH08185219}.
 \end{remark}
 \begin{remark}[Regarding the optimality of the exponents and of the asymmetries]\label{ReALL}
 Consider the case $f\equiv 1$ on $\O$ and $j:x\mapsto x$. In that case, for some constant that only depends on $|\O^\sharp|$, 
 \[ \inf_{z\in \R^d} \mathscr A(f,z)^2=c\alpha(\O)^2,\] and the Talenti inequality reduces to the torsion deficit. Since the Saint-Venant inequality is known to be quadratically stable with respect to the Fraenkel asymmetry in dimension $d\geq 2$ (and this is sharp), and linearly stable with respect to the Fraenkel asymmetry in dimension 1 (likewise, this is sharp--see  Lemma \ref{Le:SV1d} below), the stability with respect to $\alpha(\O)$ is sharp in \eqref{Eq:Talentij}--\eqref{Eq:Talentij1d}.

 Regarding the optimality of both the asymmetry $\mathscr A$ and the exponent $2$  in the dependence $f$,  consider the case $\O=\O^\sharp$ and $f=\mathds 1_E$ for some $E\subset \O$. It was proved in \cite[Proof of Proposition 2.8]{zbMATH08185219} that
 \[ \int_{\O^\sh}u_{f^\sh,\O^\sh}-\int_{\O^\sh}u_{f,\O^\sh}\geq c |E|^{\frac2d-1} \cdot |E\triangle E^\sh|^2.\]Thus, $\mathscr A^2$ is an integrated version of this estimate. Although we do not claim here that $\mathscr A$ is sharp for any $f^\sharp$, we conjecture that it is, as is the exponent 2.
 \end{remark}
 \begin{remark}
     The constant $c$ in Theorem \ref{Th:Main} is not explicit as we rely heavily on estimates from \cite{zbMATH06464861} that are non-constructive. It would be interesting to know if the recent methods of \cite{arXiv:2609.07772} could give explicit constants. We leave that as an interesting open problem.
 \end{remark}
As a consequence, we get the following:
\begin{corollary}
\label{cor:pStability}
Let $p\in [1;+\infty]$. Let $\O$ be an open set with finite volume, and let $f\in L^\infty(\O)$, $f\geq 0$, $f\not\equiv 0$. 
\begin{enumerate}
    \item In dimension $d\geq 2$, there exists a constant $c>0$ depending only on $p$, on the dimension, on $\Vert f^\sharp\Vert_{L^\infty(\O^\sharp)},\, \Vert f^\sharp\Vert_{L^1(\O^\sharp)}$ and on $|\O^\sharp|$   such that 
\begin{equation}\label{Eq:Lpstability}
\Vert u_{f^\sh,\O^\sh}\Vert_{L^p(\O^\sh)}-\Vert u_{f,\O}\Vert_{L^p(\O)}\geq c \left(\alpha(\O)^2+\inf_{z\in \R^d}\mathscr A(f,z)^2\right).
\end{equation}
\item In dimension  $1$, there exists a constant $c>0$ depending only on $p$,  on $\Vert f^\sharp\Vert_{L^\infty(\O^\sharp)},\, \Vert f^\sharp\Vert_{L^1(\O^\sharp)}$ and on $|\O^\sharp|$   such that 
\begin{equation}\label{Eq:Lpstability1}
\Vert u_{f^\sh,\O^\sh}\Vert_{L^p(\O^\sh)}-\Vert u_{f,\O}\Vert_{L^p(\O)}\geq c \left(\alpha(\O)+\inf_{z\in \R^d}\mathscr A(f,z)^2\right).
\end{equation}
\end{enumerate}


\end{corollary}

Our last result discusses comparison with $L^1$ distances and counter-examples to such comparison.
\begin{theorem}\label{Th:CE}

\begin{enumerate}
    \item\underline{Quadratic $L^1$ stability in dimension $d\geq 2$:}  Assume $d\geq 2$. Under the same assumptions as in Theorem \ref{Th:Main}, there exists a constant $c>0$ depending only on the dimension, on $j$, on $\Vert f^\sharp\Vert_{L^\infty(\O^\sharp)},\, \Vert f^\sharp\Vert_{L^1(\O^\sharp)}$ and on $|\O^\sh|$ such that, for any open $\O\subset \R^d$, there holds 
\begin{equation}\label{Eq:L1Stability}
 \int_{\O^\sh}j(u_{f^\sharp,\O^\sh})-\int_\O j(u_{f,\O})\geq c\left(\alpha(\O)^2+\inf_{z\in \R^d}\Vert f-f^\sharp(\cdot-z)\Vert_{L^1(\R^d)}^2\right),\end{equation}and the exponent is sharp.
    \item\underline{Sharp cubic $L^1$ stability in dimension $d=1$:} Assume $d=1$. Under the same assumptions as in Theorem \ref{Th:Main},  there exists a constant $c>0$ depending only on $j$, $\Vert f^\sharp\Vert_{L^\infty(\O^\sh)},\, \Vert f^\sharp\Vert_{L^1(\O^\sh)}$ and on $|\O^\sh|$ such that, for any open $\O\subset \R$, there holds  
\begin{equation}\label{Eq:Cubic}
\int_{\O^\sh}j(u_{f^\sharp,\O^\sh})-\int_\O j(u_{f,\O})\geq c\left(\alpha(\O)+\inf_{z\in \R} \Vert f-f^\sharp(\cdot-z)\Vert_{L^1(\R)}^3\right),
\end{equation} and the exponent $3$ is sharp when $j:x\mapsto x^p$. 
    \item\underline{No quadratic $L^1$ stability in dimension $d\in \{2,3\}$ if $f$ is not bounded: } Let $d\in \{2,3\}$ and $\O=\O^\sharp=\B(0;1)$. Write, for the sake of readability, $u_f=u_{f,\O^\sharp}.$ Then there exists $f\in L^2(\B(0;1))$ and a sequence   $\{f_k\}_{k\in \N}$ of functions such that  $f_k^\sharp=f^\sharp$ for every $k$, and such that, for any $p\in [1;+\infty)$
\begin{equation}\label{Eq:CEd} \frac{\Vert u_{f^\sharp}\Vert_{L^p(\B(0;1))}^p-\Vert u_{f_k}\Vert_{L^p(\B(0;1))}^p}{\inf_{z\in \R^d}\Vert f^\sharp-f_k(\cdot-z)\Vert_{L^1(\R^d)}^2}\underset{k\to \infty}\longrightarrow 0.\end{equation}
\end{enumerate}
\end{theorem}
\begin{remark}
The sharpness claimed in \eqref{Eq:L1Stability} is to be understood in the following sense: when $f=\mathds 1_E$, the squared $L^1$ distance was shown to be sharp in \cite{zbMATH08185219}. We conjecture that the squared $L^1$ distance is sharp for any $0\leq f^\sharp\in L^\infty(\O^\sharp)$ and $j$ satisfying \eqref{Eq:Hypj} but leave that as an open problem.
\end{remark}
\subsection{Bibliography and comparison with our results}\label{Se:Bib}
Quantitative inequalities have been a central topic in shape optimisation over the last two decades. Starting from the quantitative isoperimetric inequality by Fusco, Maggi \& Pratelli \cite{zbMATH05578735} a number of geometric and functional inequalities have been quantified in sharp ways. Of major importance to us is the quantitative Saint-Venant inequality, proved by Brasco, De Philippis \& Velichkov  \cite{zbMATH06464861} as a means to obtaining the quantitative Faber-Krahn inequality. We single it out here for future reference: introducing, for any open bounded set $\O$ the torsion function 
\begin{equation}\label{Eq:Torsion}
\begin{cases}-\Delta w_\O=1&\text{ in }\O,\, 
\\ w_\O \in W^{1,2}_0(\O)\end{cases}\end{equation} and 
\[ T(\O):=\int_\O w_\O,\] they prove that, in dimension $d\geq 2$, there exists a constant $c=c(d)>0$ such that, for any open set $\O$ with finite volume, there holds 
\begin{equation}\label{Eq:QuantitativeSV}
|\O|^{-1-\frac2d}\left(T(\O^\sharp)-T(\O)\right)\geq c \alpha(\O)^2.\end{equation} In dimension 1, we shall prove (see Lemma \ref{Le:SV1d}) that there exists a constant $c>0$ such that, for any open set $\O$ with finite volume, there holds 
\begin{equation}\label{Eq:QuantitativeSV1d}
|\O|^{-3}\left(T(\O^\sharp)-T(\O)\right)\geq c \alpha(\O).\end{equation}The proof, naturally, is much easier in the one-dimensional case and follows from direct computations.

Regarding Talenti inequalities, some partial results were previously obtained by several authors:
\begin{enumerate}
\item In the case $\O=\O^\sharp$ and  when $f$ is a characteristic function a sharp quantitative parabolic Talenti inequality for the $L^2$ norm was obtained by one of the authors in \cite{zbMATH07439881}.
\item In \cite{K21}, the author focuses on the case $f\equiv 1$ and shows that:
\[
\begin{split}
	&\|u_{1^\sharp,\O^\sharp}\|^p_{L^p(\Om^\sharp)}-\|u_{1,\O}\|^p_{L^p(\Om)}\geq c\alpha(\Om)^{2+p}\qquad \textrm{ if }p\in[1,+\infty),\\ &\|u_{1^\sharp,\O^\sh}\|_{L^\infty(\Om^\sharp)}-\|u_{1,\O}\|_{L^\infty(\Om)}\geq c\alpha(\Om)^{3}.
\end{split}
\]
Notice that here $1^\sharp$ is the constant function defined on the centered ball $\Om^\sharp$. Theorem \ref{Th:Main} improves these results.
\item In \cite{arXiv:2311.18617} the authors focus on the case $p=\infty$: it is shown that there exist $c=c(d,|\Om|,f^\sharp)$ and $\theta=\theta(d)$ such that
\begin{equation}\label{eq:ABMP}\|u_{f^\sharp}\|_{L^\infty(\Om^\sharp)}-\|u_f\|_{L^\infty(\Om)}\ge c\left(\alpha(\Om)^3+\inf_{x_0\in\R^d}\Big\|f-f^\sharp(\cdot+x_0)\Big\|_{L^1(\R^d)}^\theta\right).
\end{equation}
In the same paper, the authors also obtain a partial result in the case $p=2$, namely:
$$\|u_{f^\sharp}\|_{L^2(\Om^\sharp)}^2-\|u_f\|_{L^2(\Om)}^2\ge c\alpha(\Om)^4.$$
In other words, the control of the $L^2$-Talenti deficit by the asymmetry of $f$ remained open. Our results extend and improve these results.
\item In the results described in items 2 and 3, the  exponents obtained are not expected to be optimal. On the contrary, in \cite{zbMATH08185219} the first two authors obtained a sharp result in the particular case where $\Om=\B(0;1)$ and $f$ is the characteristic function of a set. More precisely it is shown that for any $p\in [1,\infty]$ and any $E\subset \B(0;1)$ with $0<|E|<|\B(0;1)|$, there exists $c=c(p,|E|)$ such that
\begin{equation}\label{eq:al}
\|u_{f^\sharp}\|_{L^p(\B(0;1))}-\|u_f\|_{L^p(\B(0;1))}\ge c\|f^\sharp-f\|^2_{L^1(\B(0;1))},
\end{equation}
where $f=\mathds 1_E$, and that the exponent 2 is sharp. This result from \cite{zbMATH08185219} is a particular case of Theorem \ref{Th:Main}: from Remark \ref{Re:AL}, we deduce that for any $f\in L^\infty(\B(0;1)),\, f\geq 0$ and any $p\in [1;+\infty]$
\[\|u_{f^\sharp}\|_{L^p(\B(0;1))}-\|u_f\|_{L^p(\B(0;1))}\ge c\int_0^\infty \mu_f(t)^{\frac2d-1}\left|\{f> t\}\triangle \{f^\sharp> t\}\right|^2dt,\] so that, when $f=\mathds 1_E$, this expression simplifies to
\[\|u_{f^\sharp}\|_{L^p(\B(0;1))}-\|u_f\|_{L^p(\B(0;1))}\ge c|E^\sharp|^{\frac2d-1}|E\triangle E^\sharp|^2.\] As the constant $c$ only depends on $|E^\sharp|$ and $p$, this gives the required result.
\end{enumerate}

\section{Proof of Theorem \ref{Th:Main}}
The proof follows several steps; this section is structured accordingly. Throughout, we assume $j$ satisfies \eqref{Eq:Hypj} and we let $\O$ be an open set with finite measure. We consider a reference configuration $f_{ref}$ that is equal to its rearrangement ($f_{ref}=f_{ref}^\sharp$).  Up to scaling we assume 
\[ |\O|=|\B(0;1)|.\]We introduce the notations
\[ \mathscr F(\O):=\left\{f\in L^\infty(\O):\, f\geq 0,\, f^\sharp=f_{ref}\right\},\]
\[ D_j(f,\O):=\int_{\O^\sharp} j(u_{f^\sharp,\O^\sharp})-\int_\O j(u_{f,\O})\] and
\[ \mathscr D_p(f,\O):=\int_{\O^\sharp}\left|u_{f^\sharp,\O^\sharp}\right|^p-\int_\O \left|u_{f,\O}\right|^p.\]
\subsection{Reduction to the $L^1$ deficit}
We begin with the following lemma: 
\begin{lemma}\label{Le:Torsion}
There exists a constant $c>0$ that only depends on $\Vert f^\sharp\Vert_{L^\infty(\O^\sharp)},\,\Vert f^\sharp\Vert_{L^1(\O^\sharp)},\, |\O^\sharp| $, $j$ and  $d$ such that 
\begin{equation} \forall f \in \mathscr F(\O),\, \, D_j(f,\O)\geq c \mathscr D_1(f,\O).\end{equation}
\end{lemma}

\begin{proof}[Proof of Lemma \ref{Le:Torsion}] The function
\[ v_f=u_{f^\sharp,\O^\sharp}-u_{f,\O}^\sharp\] is radially symmetric and decreasing \cite[Proof of Theorem 1]{zbMATH03531830}.
Observe that, in radial coordinates  (with abuse of notation we write radial functions $u=u(x)=u(|x|)$)
\[ D_j(f,\O)=c_d\int_0^1 (j(u_{f^\sharp,\O^\sharp})-j(u_{f,\O}^\sharp))r^{d-1}dr\] for some constant $c_d$. The function $j$ is monotone increasing so that (assuming without loss of generality that $j(0)=0$)
\[ j(u_{f,\O})^\sharp=j(u_{f,\O}^\sharp).\]The convexity of $j$, \eqref{Eq:Hypj} and the fact that $0\leq u_{f,\O}^\sharp\leq u_{f^\sharp,\O^\sharp}$ imply
\[ \frac{j(u_{f^\sharp,\O^\sharp})-j(u_{f,\O}^\sharp)}{u_{f^\sharp,\O^\sharp}-u_{f,\O}^\sharp}\geq \frac{j(u_{f^\sharp,\O^\sharp})-j(0)}{u_{f^\sharp,\O^\sharp}}\geq c j'(u_{f^\sharp,\O^\sharp}).\] As $v_f\geq 0$, this implies that, up to changing the constant $c>0$,
\[ D_j(f,\O)\geq c \int_0^1 j'(u_{f^\sharp,\O^\sharp})\left(u_{f^\sharp,\O^\sharp}-u_{f,\O}^\sharp\right)r^{d-1}dr.\] We recall the Chebyshev integral inequality: if $h_1,\, h_2$ are two non-increasing, non-negative functions and $\mu$ is a probability density, 
\begin{equation}\label{Eq:HLW} \int_0^1 h_1 h_2 d\mu\geq \int_0^1 h_1 d\mu\int_0^1 h_2 d\mu.\end{equation} We refer to \cite[Section 2.17]{zbMATH03073200}, but also note that this follows  from the exact expression
\[ \int_0^1 h_1 h_2 d\mu-\int_0^1 h_1d\mu \int_0^1 h_2d\mu=\frac12 \int_0^1\int_0^1 (h_1(x)-h_1(y))(h_2(x)-h_2(y))d\mu(x)d\mu(y).\] Note that $j'(u_{f^\sharp,\O^\sharp})\in L^1(\B(0;1))$ as $u_{f^\sharp,\O^\sharp}\in L^\infty(\O^\sharp)$ and $j'$ is locally bounded by convexity and monotonicity. Furthermore, since $u_{f^\sharp,\O^\sharp}$ is explicit in radial coordinates, namely, 
\[u_{f^\sh,\O^\sh}(x)=\int_{|x|}^{1}\frac{1}{t^{d-1}}
\left(\int_{0}^{t}s^{d-1}f^*(\omega_d s^d)ds\right)dt,\] so that 
\[
\frac{\Vert f^\sh\Vert_{L^1(\O^\sh)}}{\omega_d}w_{\O^\sh}
\leq u_{f^\sharp,\O^\sh}
\leq \Vert f^\sh\Vert_{L^\infty(\O^\sh)} w_{\O^\sh}.
\]Thus, constants depending only on $u_{f^\sh,\O^\sh}$ and $j$ only depend on $j$ and $\Vert f^\sh\Vert_{L^\infty(\O^\sh)},\, \Vert f^\sh\Vert_{L^1(\O^\sh)}.$

In particular, considering $d\mu(r):=c_d' r^{d-1}dr$  for some dimensional constant $c_d'$ we deduce
\[ D_j(f,\O)\gtrsim \int_0^1 \left(u_{f^\sharp,\O^\sharp}-u_{f,\O}^\sharp\right)r^{d-1}\,dr=\frac1{d\omega_d}\left(\int_{\O^\sharp}u_{f^\sharp,\O^\sharp}-\int_\O u_{f,\O}\right)\] for some constant only depending on $\Vert f^\sharp\Vert_{L^\infty(\O^\sharp)},\,\Vert f^\sharp\Vert_{L^1(\O^\sharp)},\, |\O^\sharp| $, $j$ and $d$.

\end{proof}
As a consequence, it suffices to prove Theorem \ref{Th:Main} for $j:x\mapsto x$.
A convenient rewriting of $\mathscr D_1$ is given by the following result:
\begin{lemma}\label{Le:D1Torsion}
There holds
\[ \mathscr D_1(f,\O)=\int_{\O^\sharp}f^\sharp w_{\O^\sharp}-\int_\O fw_\O\] where $w_\O$ solves \eqref{Eq:Torsion}.
\end{lemma}
\begin{proof}[Proof of Lemma \ref{Le:D1Torsion}]
Multiplying \eqref{Eq:Torsion} by $u_{f,\O}$  (and similarly on $\O^\sharp$ with $u_{f^\sharp,\O^\sharp}$) and integrating by parts yields
\[ \int_{\O^\sh}u^\sh_{f,\O}=\int_\O u_{f,\O}=\int_\O fw_\O, \quad\int_{\O^\sharp} u_{f^\sharp,\O^\sharp}=\int_{\O^\sharp} f^\sharp w_{\O^\sharp}.\] We thus conclude
\[  \mathscr D_1(f,\O)=\int_{\O^\sharp} f^\sharp w_{\O^\sharp}-\int_\O f w_\O.\]
\end{proof}

We then show that the $L^1$ deficit controls the torsion deficit.
\begin{proposition}\label{Pr:FA}
There exists a constant $c>0$ that only depends on $\Vert f^\sharp\Vert_{L^\infty(\O^\sharp)},\,\Vert f^\sharp\Vert_{L^1(\O^\sharp)},\, |\O^\sharp| $ and $d$ such that
\begin{equation}\label{Eq:LBSV}\mathscr D_1(f,\O)\geq c\left(T(\O^\sharp)-T(\O)\right).\end{equation} As a consequence, the following holds:
\begin{enumerate}
    \item In dimension $d\geq 2$, there exists a constant $c>0$ depending on the same quantities such that
\[  \mathscr D_1(f,\O)\geq c \alpha(\O)^2.\]
    \item In dimension 1, there exists a constant $c>0$ depending on the same quantities such that
\[  \mathscr D_1(f,\O)\geq c \alpha(\O).\]
\end{enumerate}
\end{proposition}
There are two parts to this proposition: the first one is \eqref{Eq:LBSV}, which reduces the $L^1$ deficit to the torsion deficit. The second part applies the quantitative Saint-Venant inequality to this torsion deficit. As explained in the introduction (see Remark \ref{Re:SV1d}), we could not locate the one-dimensional quantitative version of the Saint-Venant inequality. We thus prove it before Proposition \ref{Pr:FA}.
\begin{lemma}\label{Le:SV1d}There exists a constant $c>0$ such that, for any open set $\O\subset \R$ with finite measure, there holds \begin{equation}\label{Eq:QuantitativeSV1d1}
|\O|^{-3}\left(T(\O^\sharp)-T(\O)\right)\geq c \alpha(\O),\end{equation} and this inequality is sharp.
\end{lemma}
\begin{proof}[Proof of Lemma \ref{Le:SV1d}]
We first recall that the torsion of an interval $(a;b)$ is explicitly given by
\[ T((a;b))=c|b-a|^3\] for some constant $c$. Let $\O$ be an open set with finite measure, which we write as a countable union of disjoint open intervals
\[ \O=\cup_{j=0}^\infty I_j\] where $I_j=(a_j;b_j)$. Up to relabelling, we assume that $|I_0|=\max_{j\in \N}|I_j|$. In particular, 
\begin{align*} T(\O^\sharp)-T(\O)&= c\left(|\O^\sharp|^3-\sum_{j\in \N}|I_j|^3\right)
\\&\geq c\left(|\O^\sharp|^3-|I_0|^2 \sum_{j\in \N}|I_j|\right)
\\&\geq c |\O^\sharp|^2\left(|\O^\sharp|-|I_0|\right).\end{align*} Assume up to a translation that $I_0\subset \O^\sharp$. Then
\[ \alpha(\O)\leq \frac{|\O\triangle \O^\sharp|}{|\O|}.\] As 
\[ |\O\triangle \O^\sharp|=2\left(|\O^\sharp|-|\O^\sharp \cap \O|\right)\leq 2\left(|\O^\sharp|-|\O^\sharp\cap I_0|\right)\leq 2\left(|\O^\sharp|-|I_0|\right)\] we deduce 
\[ \alpha(\O)\leq \frac2{c}|\O^\sharp|^{-3}\left(T(\O^\sharp)-T(\O)\right).\] The sharpness of this inequality follows from considering the domain $\O_\e:=(-\e;0)\cup (\e;1)$. Thus, 
\[ T(\O_\e^\sh)-T(\O_\e)= c(1-(1-\e)^3-\e^3)\sim c'\e\] for some $c'>0$. As $\alpha(\O_\e)\sim c''\e$ for some $c''>0$,  the sharpness follows.
\end{proof}
\begin{proof}[Proof of Proposition \ref{Pr:FA}]
As $v=w_{\O^\sharp}-w_\O^\sharp$ is radially non-increasing by \cite[Theorem 1]{zbMATH03531830}, and as, by the Hardy-Littlewood inequality 
\[ \int_\O fw_\O\leq \int_{\O^\sharp}f^\sharp w_\O^\sharp,\] the Chebyshev inequality \eqref{Eq:HLW}  yields
\[ \int_{\O^\sharp} f^\sharp w_{\O^\sharp}-\int_\O f w_\O\geq \frac{\Vert f^\sharp\Vert_{L^1(\O^\sharp)}}{|\O^\sharp|}\int_{\O^\sharp} (w_{\O^\sharp}-w_\O^\sharp)=\frac{\Vert f^\sharp\Vert_{L^1(\O^\sharp)}}{|\O^\sharp|}\left(T(\O^\sharp)-T(\O)\right)\] with $T(\O)=\int_\O w_\O$. The quantitative Saint-Venant inequality \eqref{Eq:QuantitativeSV} gives the conclusion in dimension $d\geq 2$, and Lemma \ref{Le:SV1d} gives it in dimension 1.
\end{proof}

\subsection{Proof in the case of bounded domains}

We begin with a technical lemma.
\begin{lemma}\label{Le:Replacement} 
There exists a constant $c>0$ that only depends on $\Vert f^\sharp\Vert_{L^1(\O^\sharp)},\, \Vert f^\sharp\Vert_{L^\infty(\O^\sharp)}$, $d$ and $|\O^\sharp|$  such that for any $f,\, f'\in \mathscr F(\O)$ there holds
\begin{equation}\label{Eq:Transfer}
\mathscr D_1(f',\O)+\mathscr D_1(f,\O)\geq c\int_0^\infty \mu_{f}(t)^{\frac2d-1}\left|\{f> t\}\triangle \{f'> t\}\right|^2dt.
\end{equation}
\end{lemma}

\begin{proof}[Proof of Lemma \ref{Le:Replacement}]
Let us recall that from Lemma \ref{Le:D1Torsion}
\[
\mathscr D_1(f,\Om)= \int_{\O^\sh}f^\sh w_{\O^\sh}-\int_\O fw_\O
\]
We approach the desired estimate in the spirit of \cite{Lemou}. Specifically, let us first show that for any two $E_1,\, E_2\subset \O$, $|E_1|=|E_2|=\mu$, there holds
\begin{equation}\label{Eq:Goal1}
\mathscr D_1(\mathds 1_{E_1},\O)+\mathscr D_1(\mathds 1_{E_2},\O)\geq c_d \mu^{\frac2d-1}\left|E_1\triangle E_2\right|^2.
\end{equation}
Let us observe that the conclusion follows from \eqref{Eq:Goal1}. Indeed, first observe that the layer cake formula and $\mathds 1_{\{f>t\}}^\sh=\mathds 1_{\{f^\sh>t\}}$ give
\[ \mathscr D_1(f,\O)=\int_0^\infty \int_{\{f^\sharp>t\}}w_{\O^\sharp}dt-\int_0^\infty  \int_{\{f>t\}} w_\O dt=\int_0^\infty \mathscr D_1\left(\mathds 1_{\{f>t\}},\O\right)dt\]so that, applying \eqref{Eq:Goal1} for each $t$,
\[ \mathscr D_1(f',\O)+\mathscr D_1(f,\O)\geq c\int_0^\infty  \mu_f(t)^{\frac2d-1}\left|\{f>t\}\triangle \{f'>t\}\right|^2dt.\]
To prove \eqref{Eq:Goal1}, we consider 
\[ K:s\mapsto \int_0^{s} w_{\O^\sharp}^*.\]
We denote \[ a:=|E_1\setminus E_2|=|E_2\setminus E_1|=\frac{|E_1\triangle E_2|}2.\] By the Talenti inequality
\[ \mathscr D_1(\mathds 1_{E_1},\O)+\mathscr D_1(\mathds  1_{E_2},\O)\geq2K(\mu)- K(\mu-a)-K(\mu+a).\]
Indeed, we first note that by the Hardy-Littlewood inequality we have
\[\int_{E_1}w_\O+\int_{E_2}w_\O=\int_\O \mathds 1_{E_1\cap E_2}w_\O+\int_\O \mathds 1_{E_1\cup E_2}w_\O\leq \int_{(E_1\cap E_2)^\sharp} w_\O^\sharp+\int_{(E_1\cup E_2)^\sharp} w_\O^\sharp.\]
As, for any volume $V$ and any set $ E$ of volume $V$, the Hardy-Littlewood and Talenti inequalities  imply
\[ \int_{ E} w_\O\leq \int_{ E^\sharp}w_{\O^\sharp}= K(V)\] this gives 
\[\int_{E_1}w_\O+\int_{E_2}w_\O\leq K(\mu-a)+K(\mu+a).\] Consequently, 
\[\mathscr D_1(\mathds 1_{E_1},\O)+\mathscr D_1(\mathds  1_{E_2},\O)\geq2K(\mu)- K(\mu-a)-K(\mu+a). \]
Now, observe that $K$ is explicit: as $w_{\O^\sharp}=\frac{1-|x|^2}{2d}$, we obtain  
\[ K(V)=c_dV-c_d'V^{\frac2d+1}\] for some dimensional constants $c_d,\, c_d'>0$. Thus, 
\[ 2K(\mu)-K(\mu-a)-K(\mu+a)=c_d'\left((\mu+a)^{\frac2d+1}+(\mu-a)^{\frac2d+1}-2\mu^{\frac2d+1}\right)=:h(a).\]
Notice that 
\[ h(a)=c_d'\mu^{\frac2d-1}a^2 \tilde h\left(\frac{a}{\mu}\right)\] with 
\[\tilde h:s\mapsto \frac{(1+s)^{\frac2d+1}+(1-s)^{\frac2d+1}-2}{s^2}. \]A study of the function $\tilde h$ shows that it is positive on $[0;1]$ (with the continuous extension $\tilde h(0)=\frac2d(1+\frac2d)$); as $a\leq \mu$, the conclusion follows: there exists a constant $c>0$ such that
\[ \mathscr D_1(\mathds 1_{E_1},\O)+\mathscr D_1(\mathds 1_{E_2},\O)\geq c \mu^{\frac2d-1}a^2.\] The conclusion follows.
\end{proof}

In the upcoming proofs, we will be making use of the following quantitative Hardy-Littlewood inequality:
\begin{proposition}\label{Pr:HL}
There exists a constant $c>0$ that only depends on the dimension, on $|\O^\sh|$ and on $\Vert f^\sharp\Vert_{L^\infty(\O^\sharp)}$, $\Vert f^\sharp\Vert_{L^1(\O^\sharp)}$ such that 
\begin{equation}\label{Eq:HL}
\int_\O f(x)|x|^2dx-\int_{\O^\sharp} f^\sharp(x)|x|^2dx\geq c \mathscr A(f,0)^2 .\end{equation}
\end{proposition}
This inequality follows from  \cite{Lemou} (specifically, from \cite[Eq. (2.15)]{Lemou}, $\sigma(x)=| x|^2$; notice that with his notations $B_\si(\mu)=c_d\mu^{(d+2)/d}$ and $\al(t)=\be(t)=\frac12\abs{\set{f>t}\triangle\set{f^\sh>t}}$).\\

In order to prove Theorem \ref{Th:Main}, we start by proving the expected estimate when $f$ is chosen to be a maximiser of the map
\[ \Psi:\mathscr F(\O)\ni g\mapsto \int_\O gw_\O.\] 
To that end, let us begin with the following result on $\Psi$:
\begin{lemma}\label{Le:OC}
The function $\Psi$ admits a unique maximiser $f_\O$ on $\mathscr F(\O)$. Furthermore, there exists a non-decreasing function $\phi$ such that
\begin{equation}
\label{eq:f_ORearrangeswO}f_\O=\phi(w_\O).\end{equation}
\end{lemma}
\begin{proof}[Proof of Lemma \ref{Le:OC}]

Observe that, as $-\Delta w_\O>0$, $|\set{w_\O=t}|=0$ for every $t\in(0,\sup w_\O)$. In particular, the distribution function of $w_\O$ never jumps, and its generalized inverse $(w_\O)^*$ is strictly decreasing. From~\cite[Theorem 1.1]{CianFer08} (see also~\cite{Burton1989}) this implies that $\Psi$ has a unique maximiser $f_\O$, and it is in the form of~\eqref{eq:f_ORearrangeswO}.
\end{proof}
We now work with $f_\O=\phi(w_\O)$ (see Lemma \ref{Le:OC}), and we prove a quantitative version of a weighted Saint-Venant inequality that only works with the asymmetry of the source.
\begin{proposition}\label{Pr:SV}
Assume $\O\subset \B(0;R)$ and that, up to translations,  $x_\O=\frac1{|\O|}\int_\O xdx=0$. There exists $c>0$ that only depends on the dimension, $R$ and on $\Vert f^\sharp\Vert_{L^1(\O^\sharp)},\, \Vert f^\sharp\Vert_{L^\infty(\O^\sharp)}$, such that  \[\mathscr D_1(f_\O,\O)= \int_{\O^\sh} f^\sharp w_{\O^\sharp}-\int_\O f_\O w_\O\geq c\mathscr A(f_\O,0)^2.\]
\end{proposition}

\begin{proof}[Proof of Proposition \ref{Pr:SV}]
From Proposition \ref{Pr:HL} there holds 
\[ \int_\O f_\O(x)|x|^2dx-\int_{\O^\sharp} f^\sharp(x)|x|^2dx\gtrsim \mathscr A(f_\O,0)^2.\] As we also have \eqref{Eq:LBSV} from \Cref{Pr:FA}, to establish the result, it suffices  to show that 
\begin{equation}\label{Eq:PCW}
\int_\O f_\O(x)|x|^2dx-\int_{\O^\sharp} f^\sharp(x)|x|^2dx\lesssim \int_{\O^\sharp} w_{\O^\sharp}-\int_\O w_\O.\end{equation} This estimate is now the main focus.  From Lemma \ref{Le:OC}, $f_\O$ can be written  $f_\O=\phi(w_\O)$ for some non-decreasing $\phi$. Let $\Phi:=\int_0^\cdot \phi$. On the one hand, we have 
\[ \int_\O f_\O |\n w_\O|^2=\int_\O \Phi(w_\O).\] On the other, 
\[ \int_\O f_\O \langle x,\n w_\O\rangle=-d\int_\O \Phi(w_\Omega).\] Consequently, we deduce a Feldman-type estimate \cite[Proposition 5]{zbMATH06898789}
\begin{align*}
\int_\O f_\O \left|\n w_\O+\frac{x}{d}\right|^2&=\int_\O f_\O |\n w_\O|^2+\frac2d\int_\O f_\O \langle x,\n w_\O\rangle+\frac1{d^2}\int_\O f_\O |x|^2
\\&=\frac1{d^2}\int_\O f_\O |x|^2-\int_\O \Phi(w_\O).
\end{align*} 
We then observe that the proof of \cite[Theorem 1]{zbMATH03531830} actually allows us to derive the pointwise bound
\begin{equation}\label{Eq:StgTalenti} |\n w_\O^\sharp|(x)\leq \frac{|x|}d=\abs{\na w_{\O^\sh}}(x).\end{equation}
Indeed, to get \eqref{Eq:StgTalenti}, proceed as in \cite{zbMATH03531830} and, fixing $t>0$, integrate the equation 
\[-\Delta w_\O=1\] on $\{w_\O>t\}$. Letting $\mu_{w_\O}$ denote the distribution function of $w_\O$, by the divergence theorem we deduce for a.e. $t$ that 
\[ \int_{\partial\{w_\O>t\}}|\n w_\O|\leq \mu_{w_\O}(t).\]Recall that 
\[ |\mu_{w_\O}'|(t)=\int_{\partial\{w_\O>t\}}\frac1{|\n w_\O|}. \]By the Cauchy-Schwarz and isoperimetric inequalities we obtain the differential inequality
\[ d^2\omega_d^{\frac{2}d}\mu_{w_\O}(t)^{2-\frac{2}{d}}\leq |\mu_{w_\O}'|(t)\mu_{w_\O}(t). 
\]In other words,
\[ d^2\omega_d^{2/d}\mu_{w_\O}(t)^{1-2/d}\leq|\mu_{w_\O}'|(t).\] Differentiating the relation  $w_\O^\sharp\left(\left(\frac{\mu_{w_\O}(s)}{\omega_d}\right)^{\frac1d}\right)=s$ yields the conclusion.

Coming back to the Feldman type estimate above, we deduce
\[ d \int_\O \Phi(w_\O)=d\int_{\O^\sharp}\Phi(w_{\O}^\sharp)=-\int_{\O^\sharp} \phi(w_\O^\sharp)\langle x,\n w_{\O}^\sharp\rangle\leq \frac1d\int_{\O^\sharp}|x|^2 f^\sharp(x)dx .\] Thus, using the fact that $\Phi(w_\O)^\sh=\Phi(w_\O^\sh)$ as $\Phi$ is non-decreasing
\[d^2\int_\O f_\O \left|\n w_\O+\frac{x}d\right|^2\geq \int_{\O} |x|^2 f_\O(x)dx-\int_{\O^\sharp} |x|^2f^\sharp(x)dx. \]
We can now expand the left-hand side to deduce, for a constant $C$ that might change from line to line but only depends on $\Vert f^\sh\Vert_{L^\infty(\O^\sh)},\, d$ and $|\O^\sh|$
\begin{align*} \int_\O f_\O   \left|\n w_\O+\frac{x}d\right|^2
&\leq \Vert f^\sharp\Vert_\infty\left(\int_\O |\n w_\O|^2+\frac1{d^2}\int_\O |x|^2+\frac2d\int_\O  \langle x,\n w_\O\rangle\right)
\\&\leq C \left(\int_\O w_\O-2\int_\O w_\O +\frac1{d^2}\int_\O |x|^2dx\right)
\\&\leq C \left(-\int_\O w_\O + \frac1{d^2}\int_\O |x|^2 dx\right).\end{align*} As 
\[ \frac1{d^2} \int_{\O^\sharp}|x|^2dx=\int_{\O^\sharp} |\n w_{\O^\sharp}|^2=\int_{\O^\sh}w_{\O^\sh}\] by the explicit expression of the torsion function, this implies 
\[ \int_\O f_\O   \left|\n w_\O+\frac{x}d\right|^2\lesssim \left(\int_{\O^\sharp}w_{\O^\sharp}-\int_\O w_\O+\int_\O |x|^2dx-\int_{\O^\sharp}|x|^2dx\right).\] We finally write 
\[ \int_{\O}|x|^2dx-\int_{\O^\sharp}|x|^2dx=\int_{\O\triangle \O^\sharp} \left||x|^2-1\right|dx\lesssim \int_{\O\triangle \O^\sharp} (1+\sup_\O|\cdot|)|1-|x||dx.\] 

Let us now introduce
\[ \beta(\O)=\int_{\O\triangle \O^\sh} |1-|x||dx\] and prove that 
\begin{equation}\label{Eq:BT} \beta(\O)\leq C \left(T(\O^\sharp)-T(\O)\right)\end{equation} for any $\O\subset \B(0;R)$ for some constant $C$ that only depends on $R$.  We need to distinguish between the dimensions:
\begin{enumerate}
     \item 
In dimension 1, this follows by Lemma \ref{Le:SV1d}, which establishes that for some constant $c>0$ there holds
\[ c\alpha(\O)\leq T(\O^\sh)-T(\O).\] We now claim that 
\[ \beta(\O)\leq |\O| R\alpha(\O). \] Indeed, assuming for notational simplicity that $|\O^\sharp|=2$, we have
\[ \beta(\O)=\int_\O |x|dx-1.\]Let $z\in \R$ and consider the interval $(z-1;z+1)$ and, by symmetry, assume that $z\geq 0$. As $x_\O=0$, 
\[ \int_\O |x|dx-1=2\int_{\O\cap (-\infty;0)}(-x)dx-1,\]whence 
\begin{align*}
\beta(\O)&=2\int_{\O\cap (-\infty;0)}(-x)dx-1
\\&\leq 2\int_{(-1;0)}(-x)dx+2R|\O\cap (-\infty;-1)|-1
\\&=2R|\O\cap (-\infty;-1)|
\\&\leq 2R |\O\setminus (z-1;z+1)|.
\end{align*}
Taking the infimum in $z$ gives the conclusion. 
\item If $d\geq 2$
, by \cite[Theorem 4.3]{zbMATH06464861}, there exists $\beta'>0$ such that, if  
\[ \beta(\O):=\int_{\O\triangle \O^\sharp} \left|1-|x|\right|dx\leq \beta'\] then
\[ \int_{\O\triangle\O^\sharp}|1-|x||dx\lesssim \int_{\O^\sharp}w_{\O^\sharp}-\int_\O w_\O\] for some constant that only depends on $R$. On the other hand, if $\beta(\O)\geq \beta'$, we know from \cite[Lemma 5.2]{zbMATH06464861} that
\[ T(\O^\sharp)-T(\O)\gtrsim \alpha(\O)^4\gtrsim \beta(\O)^4\gtrsim \beta'^3\beta(\O).\] The constant depends on $R$. 
\end{enumerate}At any rate, we deduce, using both cases, that 
\[ \beta(\O)\lesssim T(\O^\sharp)-T(\O)\] for any $\O\subset \B(0;R)$, so that \eqref{Eq:BT} is established.

This concludes the proof, as this yields
\[  \int_{\O} |x|^2 f_\O(x)dx-\int_{\O^\sharp} |x|^2f^\sharp(x)dx\lesssim \int_\O f_\O \left|\n w_\O+\frac{x}d\right|^2\lesssim \int_{\O^\sharp} w_{\O^\sharp}-\int_{\O}w_\O.\] The conclusion follows.
\end{proof}

We are now in a position to conclude in the case of uniformly bounded domains.
\begin{proposition}\label{Pr:ProofR}
Let $R>0$ and assume that $\O\subset \B(0;R)$. Assume $x_\O=\frac1{|\O|}\int_\O xdx=0$. There exists a constant $c>0$ that only depends on $d$, $|\O|$, $R$, $\Vert f^\sharp\Vert_{L^1(\O^\sharp)}$ and $\Vert f^\sharp\Vert_{L^\infty(\O^\sharp)}$ such that
\[ \mathscr D_1(f,\O)\geq c\mathscr A(f,0)^2.\]\end{proposition}
\begin{proof}[Proof of Proposition \ref{Pr:ProofR}] By the triangle inequality, for some constant that only depends on $d$, $|\O|$, $\Vert f^\sharp\Vert_{L^1(\O^\sharp)}$ and $\Vert f^\sharp\Vert_{L^\infty(\O^\sharp)}$, 
\[ \mathscr A(f,0)^2\leq c\left( \int_0^\infty \mu_f(t)^{\frac2d-1}\left|\{f>t\}\triangle\{f_\O>t\}\right|^2dt+\mathscr A(f_\O,0)^2\right).\]
By Proposition \ref{Pr:SV}, we deduce that, up to replacing $c$ with another constant,
\[ \mathscr A(f,0)^2\leq c\left(  \int_0^\infty \mu_f(t)^{\frac2d-1}\left|\{f>t\}\triangle\{f_\O>t\}\right|^2dt+\mathscr D_1(f_\O,\O)\right).\] However, by definition of $f_\O$, 
\[\mathscr D_1(f_\O,\O)\leq \mathscr D_1(f,\O).\] Thus, Lemma \ref{Le:Replacement} gives
\[ \mathscr A(f,0)^2\leq C \mathscr D_1(f,\O).\] The conclusion follows. 
\end{proof}

\subsection{Reduction to uniformly bounded domains}
Throughout, we fix $\O$ open and such that $|\O|=|\B(0;1)|.$
\subsubsection{The case $d\geq 2$}
We now assume that $d\geq 2$.
We recall the following result from \cite[Proof of Lemma 5.3, (5.15)--(5.16)]{zbMATH06464861}:
\begin{lemma}\label{Le:BDPV}
There exist $R=R(d),\, \delta=\delta(d)>0$ and a constant $C$ only depending on $d$  such that, if  $T(\O^\sharp)-T(\O)\leq \delta$, then there exists an open  $U\subset \O$ with $|U|\geq \frac{|\O|}2$, $\mathrm{diam}(U)\leq R$ and that satisfies 
\[
|\O\setminus U|\leq C(T(\O^\sharp)-T(\O)),\, T(\O)-T(U)\leq C(T(\O^\sharp)-T(\O)).\] \end{lemma}

\begin{proof}[Proof of Theorem \ref{Th:Main} when $d\geq 2$] We have already noticed (Lemma \ref{Le:Torsion}) that it suffices to prove the theorem for $j:x\mapsto x$. To conclude the proof, we  argue by contradiction and assume that there exists a sequence $\{\O_k\}_{k\in \N}$ and $\{f_k\}_{k\in \N}$ such that for any $k\in \N$ $f_k^\sharp=f_{ref}$ and $\O_k^\sharp=\O^\sharp=\B(0;1)$, and such that 
\begin{equation}\label{Eq:HypContrad} \frac{\mathscr D_1(f_k,\O_k)}{\inf_{z\in \R^d}\mathscr A(f_k,z)^2}\underset{k\to \infty}\rightarrow 0.\end{equation} From Proposition \ref{Pr:FA} it follows that 
\begin{equation}\label{Eq:Boat}T(\O^\sharp)-T(\O_k)\underset{k\to \infty}\rightarrow 0.\end{equation} Let $R,\, \delta$ be given by Lemma \ref{Le:BDPV}.  If $\O_k\subset \B(0;R)$ for any $k$, it suffices to apply Propositions \ref{Pr:FA}--\ref{Pr:ProofR} to conclude. Else, let, for any $k$, $U_k$ be given by Lemma \ref{Le:BDPV}, so that  
\[ |\O_k\setminus U_k|\leq c_d(T(\O^\sharp)-T(\O_k)),\, T(\O_k)-T(U_k)\leq c_d (T(\O^\sharp)-T(\O_k)).\] 
We now define 
\[ \tilde f_k:=f_k\mathds 1_{U_k}.\]
Let us show the following three estimates:
 first, there exists $c>0$ such that for any $k$ large enough
 \begin{equation}\label{Eq:NormComparison}
 c\Vert f_k\Vert_{L^1(\O_k)}\leq \Vert \tilde f_k\Vert_{L^1(U_k)},\,  c\Vert f_k\Vert_{L^\infty(\O_k)}\leq \Vert \tilde f_k\Vert_{L^\infty(U_k)},
 \end{equation}
 second, \begin{equation}\label{Eq:LastEstimate}
\mathscr D_1(\tilde f_k,U_k)\leq \mathscr D_1(f_k,\O_k)+\Vert f^\sharp\Vert_{L^\infty(\O^\sharp)}(T(\O_k)-T(U_k))
\end{equation}
and, third, for any $z\in \R^d$,
\begin{equation}\label{Eq:LastEstimate2} \mathscr A(f_k,z)^2\leq C \left(\mathscr A(\tilde f_k,z)^2+ T(\O^\sharp)-T(\O_k)\right).\end{equation}
Assuming these three inequalities are established, the conclusion of the proof goes as follows: we first apply  Proposition \ref{Pr:ProofR}; the constants given in that proposition only depend on $\Vert \tilde f_k\Vert_{L^1(U_k)}$ and $\Vert \tilde f_k\Vert_{L^\infty(U_k)}$, so that from \eqref{Eq:NormComparison} there exists a constant $C$ that only depends on $\Vert f_k\Vert_{L^\infty(\O_k)}$ and $\Vert f_k\Vert_{L^1(\O_k)}$ such that, for any $k\in \N$, letting $z_k:=\frac1{|U_k|}\int_{U_k}xdx$
\[ \mathscr A(\tilde f_k,z_k)^2\leq C \mathscr D_1(\tilde f_k,U_k).\]
From \eqref{Eq:LastEstimate} this gives
\[\mathscr A(\tilde f_k,z_k)^2\leq C \left(\mathscr D_1(f_k,\O_k)+T(\O_k)-T(U_k)\right)\] so that, since 
$T(\O_k)-T(U_k)\leq c_d(T(\O^\sharp)-T(\O_k)),$
\[\mathscr A(\tilde f_k,z_k)^2\leq C \left(\mathscr D_1(f_k,\O_k)+T(\O^\sharp)-T(\O_k)\right).\] Consequently, as we already know from Proposition \ref{Pr:FA} that \[ T(\O^\sharp)-T(\O_k)\leq C \mathscr D_1(f_k,\O_k),\] we obtain 
\[ \mathscr A(\tilde f_k,z_k)^2\leq C \mathscr D_1(f_k,\O_k).\] Using \eqref{Eq:LastEstimate2}, we finally deduce 
\[  \mathscr A(f_k,z_k)^2\leq C \left(\mathscr A(\tilde f_k,z_k)^2+ T(\O^\sharp)-T(\O_k)\right)\leq C\mathscr D_1(f_k,\O_k),\] in contradiction with \eqref{Eq:HypContrad}. This concludes the proof.

Let us now give the proofs of the three estimates \eqref{Eq:NormComparison}--\eqref{Eq:LastEstimate}--\eqref{Eq:LastEstimate2}. We begin with \eqref{Eq:NormComparison} with the $L^1$ norm. Observe that from 
\eqref{Eq:Boat} and the estimate 
$|\O_k\setminus U_k|\leq c_d(T(\O^\sharp)-T(\O_k))$ we obtain 
\[ 0\leq \Vert f_k\Vert_{L^1(\O_k)}-\Vert \tilde f_k\Vert_{L^1(U_k)}\leq \Vert f_k\Vert_{L^\infty}|\O_k\setminus U_k|\leq c_d\Vert f_k\Vert_{L^\infty(\O_k)}(T(\O^\sharp)-T(\O_k))\underset{k\to \infty}\rightarrow 0.\] In particular, for any $k$ large enough so that 
\[ c_d\Vert f_k\Vert_{L^\infty(\O_k)}(T(\O^\sharp)-T(\O_k))\leq \frac12\Vert f_k\Vert_{L^1(\O_k)}=\frac12\Vert f^\sharp\Vert_{L^1(\O^\sharp)}\] there holds 
\[ \Vert f_k\Vert_{L^1(\O_k)}\leq \Vert \tilde f_k\Vert_{L^1(U_k)}+\frac12\Vert f_k\Vert_{L^1(\O_k)},\] whence the conclusion. As regards the $L^\infty$ norm, observe that 
\[ \left|\left\{\tilde f_k\geq \frac{\Vert f_k\Vert_{L^\infty(\O_k)}}2\right\}\right|= \left|\left\{\tilde f_k\geq \frac{\Vert f^\sharp\Vert_{L^\infty(\O^\sharp)}}2\right\}\right|\geq \mu_{f^\sharp}\left(\frac{\Vert f^\sharp\Vert_{L^\infty(\O^\sharp)}}2\right)-|\O_k\setminus U_k|.\] For $k$ large enough we thus deduce from $|\O_k\setminus U_k|\underset{k\to \infty}\rightarrow 0$ that  $\left|\left\{\tilde f_k\geq \frac{\Vert f_k\Vert_{L^\infty(\O_k)}}2\right\}\right|>0$ and so 
\[ \Vert \tilde f_k\Vert_{L^\infty(U_k)}\geq  \frac{\Vert f_k\Vert_{L^\infty(\O_k)}}2.\] This concludes the proof of \eqref{Eq:NormComparison}.

Regarding \eqref{Eq:LastEstimate} we observe that as $U_k\subset \O_k$, $w_{U_k}\leq w_{\O_k}$, $U_k^\sharp\subset \O^\sharp$, $w_{U_k^\sharp}\leq w_{\O^\sharp}$. In particular, 
\begin{align*}
\mathscr D_1(\tilde f_k,U_k)&=\int_{U_k^\sharp}\tilde f_k^\sharp w_{U_k^\sharp}-\int_{U_k}f_k w_{U_k}
\\&\leq \int_{\O^\sharp} f^\sharp w_{\O^\sharp}-\int_{\O_k}f_k w_{\O_k}+\int_{\O_k}f_k w_{\O_k}-\int_{\O_k} f_k\mathds 1_{U_k}w_{U_k}
\\&\leq \mathscr D_1(f_k,\O_k)+\int_{\O_k} f_k\left(w_{\O_k}-w_{U_k}\right) \text{ where we extend $w_{U_k}$ by zero outside of $U_k$}
\\&\leq  \mathscr D_1(f_k,\O_k)+\Vert f^\sharp\Vert_{L^\infty(\O^\sharp)} \cdot \int_{\O_k}(w_{\O_k}-w_{U_k})
\\&=\mathscr D_1(f_k,\O_k)+\Vert f^\sharp\Vert_{L^\infty(\O^\sharp)}\cdot\left(T(\O_k)-T(U_k)\right).
\end{align*}
This proves \eqref{Eq:LastEstimate}.

It remains to prove \eqref{Eq:LastEstimate2}.  It suffices to establish
\begin{equation}\label{Eq:LE3}
\mathscr A(f_k,z)^2\leq C \left(\mathscr A(\tilde f_k,z)^2+\Vert f_k-\tilde f_k\Vert_{L^1(\O_k)}\right)
\end{equation}as the estimate $\Vert f_k-\tilde f_k\Vert_{L^1(\O_k)}\leq \Vert f^\sharp\Vert_{L^\infty(\O^\sharp)}\cdot |\O_k\setminus U_k|\leq c_d \Vert f^\sharp\Vert_{L^\infty(\O^\sharp)}\left(T(\O^\sharp)-T(\O_k)\right)$ then allows us to conclude. To prove \eqref{Eq:LE3} fix $t\geq 0$ and 
\[ E:=\{f_k> t\},\, F:=\{\tilde f_k> t\}\subset E.\] Letting $z\in \R^d$ be fixed and $r_E$ (resp. $r_F$) be such that 
\[ |\B(z;r_E)|=|E|,\text{ resp. }|\B(z;r_F)|=|F|\] we obtain 
\[ |E\triangle \B(z;r_E)|\leq |F\triangle \B(z;r_F)|+|\B(z;r_E)\triangle \B(z;r_F)|+|E\triangle F|\leq |F\triangle \B(z;r_F)|+2(|E|-|F|).\]
As $d\geq 2$, $\frac2d-1\leq 0$, whence $|E|^{\frac2d-1}\leq |F|^{\frac2d-1}$. We deduce that for some constant $C$
\begin{align*} |E|^{\frac2d-1}|E\triangle \B(z;r_E)|^2&\leq C|E|^{\frac2d-1}\left(|F\triangle \B(z;r_F)|^2+(|E|-|F|)^2\right)
\\&\leq C |F|^{\frac2d-1}|F\triangle \B(z;r_F)|^2
\\&+C|E|^{\frac2d}(|E|-|F|)
\\&\leq C |F|^{\frac2d-1}|F\triangle \B(z;r_F)|^2+C|\O^\sharp|^{\frac2d}\left(|E|-|F|\right).\end{align*} Integrating this estimate in $t$ gives, for some constant depending on $|\O^\sharp|$ only, 
\[ \mathscr A(f_k,z)^2\leq C \left(\mathscr A(\tilde f_k,z)^2+\Vert f_k-\tilde f_k\Vert_{L^1(\O_k)}\right),\] whence the conclusion.
\end{proof}

\subsubsection{The case $d=1$}
We cannot use \cite[Proof of Lemma 5.3, (5.15)--(5.16)]{zbMATH06464861} directly when $d=1$; nevertheless, we have the following result:
\begin{lemma}\label{Le:BDPV1d}
Assume $d=1$ and let $\O$ be an open set with finite measure. Write $\O=\cup_{j\in \N}I_j$ where each $I_j$ is an open interval, $I_j\cap I_\ell=\emptyset$ whenever $j\neq \ell$. Up to reordering, assume that $|I_0|=\max_{j\in \N}|I_j|$ and set $U:=I_0$. Then, for some constant $C>0$ that only depends on $|\O|$,  
there holds
\[
|\O\setminus U|\leq C(T(\O^\sharp)-T(\O)),\, T(\O)-T(U)\leq C(T(\O^\sharp)-T(\O)).\]
\end{lemma}
Setting the proof of this lemma aside for the moment, we can conclude the proof of Theorem \ref{Th:Main}.
\begin{proof}[Proof of Theorem \ref{Th:Main} when  $d=1$] We argue by contradiction, taking a sequence $\{\O_k\}_{k\in \N}$ contradicting the conclusion of the theorem, taking $U_k$ given by Lemma \ref{Le:BDPV1d} and setting $\tilde f_k:=f_k\mathds 1_{U_k}$. Note that from the estimates of Lemma \ref{Le:BDPV1d}, $|U_k|\underset{k\to \infty}\rightarrow |\O^\sharp|$ and thus, for $k$ large enough, $|U_k|\in \left(\frac{|\O_k|}2;|\O_k|\right].$
The conclusion of the proof of the theorem is identical to the case $d\geq 2$; the only point that differs is the proof of Estimate \eqref{Eq:LastEstimate2}. To prove that this estimate still holds in the $d=1$ case, take $t\geq 0$ and 
\[ E:=\{f_k\geq t\},\, F:=\{\tilde f_k\geq t\}\subset E.\] Letting $z\in \R$ be fixed and $r_E$ (resp. $r_F$) be such that 
\[ |\B(z;r_E)|=|E|,\text{ resp. }|\B(z;r_F)|=|F|\] we have 
\[ |E\triangle \B(z;r_E)|\leq|F\triangle \B(z;r_F)|+2(|E|-|F|).\] In particular, up to a constant $c$ that might change from line to line
\begin{align*}
|E|\cdot |E\triangle \B(z;r_E)|^2&\leq c\left( |F|\cdot |F\triangle \B(z;r_F)|^2+(|E|-|F|)|F\triangle \B(z;r_F)|^2+ |E|\cdot (|E|-|F|)^2\right)
\\&\leq c\left(|F|\cdot |F\triangle \B(z;r_F)|^2+|\O^\sharp |^2\cdot( |E|-|F|)\right).
\end{align*}
Integrating this inequality in $t$ yields, for some constant $c$ that only depends on $|\O_k|$, the estimate \eqref{Eq:LastEstimate2}.

\end{proof}
It thus remains to prove Lemma \ref{Le:BDPV1d}:
\begin{proof}[Proof of Lemma \ref{Le:BDPV1d}]
The proof is essentially the same as that of Lemma \ref{Le:SV1d}. Recall that in the proof of Lemma \ref{Le:SV1d} we established 
\begin{align*} T(\O^\sharp)-T(\O) 
&\geq  c |\O^\sharp|^2\left(|\O^\sharp|-|U|\right)\end{align*} whence 
\[ T(\O^\sharp)-T(\O)\geq c|\O\setminus U|\]for some constant $c$ only depending on $\O^\sharp$.

Second, for some constant $c$ that might change from line to line,
\begin{align*}
T(\O)-T(U)&=c \left(\sum_{j\in \N,\, j\neq 0}|I_j|^3\right)
\\&\leq c |U|^2\sum_{j\neq 0}|I_j|
\\&=c|U|^2(|\O|-|U|)
\\&\leq c|\O|^2\left(|\O|-|U|\right)
\\&\leq c |\O|\cdot(T(\O^\sharp)-T(\O))
\end{align*}where we used the previous inequality in the last step.

\end{proof}

\section{Proof of further results}
\subsection{Proof of Corollary~\ref{cor:pStability}}
Let $p\in [1;+\infty)$. It follows from Theorem \ref{Th:Main} applied to $j:x\mapsto |x|^p$ that 
\[ \Vert u_{f^\sharp,\O^\sharp}\Vert_{L^p(\O^\sharp)}^p-\Vert u_{f,\O}\Vert_{L^p(\O)}^p\geq c\left(\alpha(\O)^{\beta(d)}+\inf_{z\in \R^d}\mathscr A(f,z)^2\right)\] where the exponent $\beta$ is defined as
\[ \beta(1)=1,\, \beta(d)=2\text{ if }d\geq 2.\]
From the convexity inequality 
\[ 1-x^p\leq p(1-x),\, x\geq 0\] we deduce that
\[\Vert u_{f^\sharp,\O^\sharp}\Vert_{L^p(\O^\sharp)}^p\left(1-\left(\frac{\Vert{u_{f,\O}}\Vert_{L^p(\O)}}{\Vert u_{f^\sharp,\O^\sharp}\Vert_{L^p(\O^\sharp)}}\right)^p\right)\geq  c\left(\alpha(\O)^{\beta(d)}+\inf_{z\in \R^d}\mathscr A(f,z)^2\right).
\] This gives the conclusion.
In the case $p=\infty$, we observe that, since 
\[ u_{f^\sharp,\O^\sharp}-u_{f,\O}^\sharp\] is symmetric decreasing, 
\[ \Vert u_{f^\sharp,\O^\sharp}\Vert_{L^\infty(\O^\sharp)}-\Vert u_{f,\O}\Vert_{L^\infty(\O)}=\left(u_{f^\sharp,\O^\sharp}-u_{f,\O}^\sharp\right)(0)\geq \frac1{|\O|}\int_{\O^\sharp}\left(u_{f^\sharp,\O^\sharp}-u_{f,\O}^\sharp\right).\] It suffices to apply Theorem \ref{Th:Main} with $j:x\mapsto x$.

\subsection{Proof of Theorem \ref{Th:CE}}
\begin{enumerate}
\item\underline{Quadratic $L^1$ stability in dimension $d\geq 2$:}
The proof of \eqref{Eq:L1Stability} follows from Theorem \ref{Th:Main} and the following lemma: \begin{lemma}\label{Le:Comparison}
Let $d\geq 2$ and $z\in \R^d$. Then, for some constant $C$ that only depends on $f^\sharp$, 
\[\Vert f-f^\sharp(\cdot-z)\Vert_{L^1}^2\leq C\mathscr A(f,z)^2.\]
\end{lemma}
\begin{proof}[Proof of Lemma \ref{Le:Comparison}]
We recall that from the layer cake  formula
\[ \Vert f-f^\sharp(\cdot-z)\Vert_{L^1}=\int_0^{\Vert f^\sh\Vert_{L^\infty(\O^\sharp)}} \left|\{f> t\}\triangle \{f^\sharp(\cdot-z)> t\}\right|dt.\] Thus, 
\begin{align*}
\Vert f-f^\sharp(\cdot-z)\Vert_{L^1}^2&\leq\int_0^{\Vert f^\sh\Vert_{L^\infty(\O^\sharp)}} \mu_f(t)^{1-\frac{2}d}dt \int_0^{\Vert f^\sh\Vert_{L^\infty(\O^\sharp)}} \mu_f(t)^{\frac2d-1}\left|\{f> t\}\triangle \{f^\sharp(\cdot-z)> t\}\right|^2dt
\\&\leq |\O|^{1-\frac{2}d}\Vert f^\sharp\Vert_{L^\infty}\mathscr A(f,z)^2.
\end{align*}
Here, we used the fact that $d\geq 2$ implies $\mu_f(t)^{1-\frac2d}\leq |\O^\sharp|^{1-\frac2d}$.
\end{proof} 
\item\underline{Sharp cubic $L^1$ stability in dimension $d=1$:}
To prove \eqref{Eq:Cubic}  it suffices to show that for any $z\in \R$
\begin{equation}\label{Eq:DTL} \mathscr A(f,z)^2\geq c \Vert f-f^\sharp(\cdot-z)\Vert_{L^1}^3\end{equation} for some constant $c$ that only depends on $\Vert f^\sharp\Vert_{L^\infty},\, \Vert f^\sharp\Vert_{L^1}.$ To prove \eqref{Eq:DTL}, let $z\in \R$ be fixed and define $\alpha(t):=|\{f>t\}\triangle \{f^\sharp(\cdot-z)> t\}|$. Thus
\begin{align*}
        \Vert f^\sharp-f(\cdot-z)\Vert_{L^1(\R)}=\int_0^{\Vert f^\sh\Vert_{L^\infty(\O^\sharp)}} \alpha
        &\leq C_d \int_0^{\Vert f^\sh\Vert_{L^\infty(\O^\sharp)}} \alpha(t)^{\frac23} \mu_f(t)^{\frac13}dt\text{ as, for some dimensional constant $C_d$, }\alpha(t)\leq C_d\mu_f(t)
        \\&=C_d\int_0^{\Vert f^\sh\Vert_{L^\infty(\O^\sharp)}} \alpha(t)^{\frac23} \mu_f(t)^{\frac{2-d}{3d}}\mu_f(t)^{\frac{2(d-1)}{3d}}dt
        \\&\leq C_d \left(\int_0^{\Vert f^\sh\Vert_{L^\infty(\O^\sharp)}} \alpha(t)^{2}\mu_f(t)^{\frac{2-d}d}dt\right)^{\frac13}\cdot\left(\int_0^{\Vert f^\sh\Vert_{L^\infty(\O^\sharp)}} \mu_f(t)^{\frac{d-1}d}dt \right)^{\frac{2}{3}}
        \\&\text{ from the $L^3-L^{\frac32}$ H\"{o}lder inequality}
        \\&\leq C(f^\sharp) \left(\int_0^{\Vert f^\sh\Vert_{L^\infty(\O^\sharp)}} \alpha(t)^{2}\mu_f(t)^{\frac{2-d}d}dt\right)^{\frac13},
    \end{align*}
    whence the conclusion. 

Let us prove that the exponent 3 is optimal by working in $\O=\O^\sharp=(-1;1)$. Let, for $\gamma \in (0,1)$, 
\begin{equation}\label{Eq:fg} f_\gamma:x\mapsto 1-|x|^\gamma.\end{equation} We define, for any $\delta >0$ small enough, the function \begin{equation}\label{Eq:fgd} f_{\gamma,\delta}:x\mapsto \begin{cases}
f_\gamma(x+\delta)\text{ if }x\in [0;\delta),
\\ f_\gamma(x-\delta) \text{ if }x\in [\delta;2\delta),
\\ f_\gamma(x)\text{ else.}
\end{cases}\end{equation} We extend $f_{\gamma}$ and $f_{\gamma,\delta}$ by zero outside of $(-1;1)$. We claim that
\begin{equation}\label{Eq:fgamma}
\inf_{z\in \R}\Vert f_{\gamma,\delta}-f_{\gamma}(\cdot-z)\Vert_{L^1}=\Vert f_{\gamma,\delta}-f_\gamma\Vert_{L^1}=C\delta^{1+\gamma}.
\end{equation}
If the minimality of $z=0$ is established, observe that
\begin{align*}
\Vert f_{\gamma,\delta}-f_\gamma\Vert_{L^1}&=2\int_0^\delta ((x+\delta)^\gamma-x^\gamma)dx
=2\delta^{1+\gamma}\int_0^1\left((1+y)^\gamma-y^\gamma\right)dy
= C\delta^{1+\gamma}.
\end{align*} 
To prove that $z=0$ is a global optimiser, we use once more the layer cake  formula. Let, for any $t\in (0;1)$, $r(t):=(1-t)^{\frac1\gamma}$ so that 
\[ E_{r,\delta}:=\{f_{\gamma,\delta}>t\}, I_r(0)=(-r(t);r(t)),\, I_r(z):=\{f_\gamma(\cdot-z)>t\}=(z-r;z+r).\] 
Let $t_\delta:=1-(2\delta)^\gamma$. If $t\in (0;t_\delta)$, 
\[ E_{r(t),\delta}=(-r(t);r(t))\] and so 
\[ \left|E_{r(t),\delta}\triangle I_{r(t)}(z)\right|=2\min \left(|z|,2r(t)\right).\] As $t\in (0;t_\delta)$, $r(t)\geq 2\delta$ and so 
\begin{equation}\label{Eq:I1}
\left|E_{r(t),\delta}\triangle I_{r(t)}(z)\right|\geq 2 \min \left(|z|,4\delta\right).
\end{equation}
If $t\in (t_\delta;1)$, we can use the triangle inequality to obtain 
\[\left|E_{r(t),\delta}\triangle I_{r(t)}(z)\right|-\left|E_{r(t),\delta}\triangle I_{r(t)}(0)\right|\geq -\left|I_{r(t)}(z)\triangle I_{r(t)}(0)\right|=-2\min\left(|z|,2r(t)\right).\] As $r(t)\leq 2\delta$, this yields 
\begin{equation}\label{Eq:I2}\left|E_{r(t),\delta}\triangle I_{r(t)}(z)\right|-\left|E_{r(t),\delta}\triangle I_{r(t)}(0)\right|\geq -2\min\left(|z|,4\delta\right).\end{equation} Summing  \eqref{Eq:I1}--\eqref{Eq:I2} gives 
\[ \int_0^1 \left(\left|E_{r(t),\delta}\triangle I_{r(t)}(z)\right|-\left|E_{r(t),\delta}\triangle I_{r(t)}(0)\right|\right)dt\geq 2\min (|z|,4\delta)\left(2t_\delta-1\right).\] As $2 t_\delta-1>0$ if $\delta$ is small enough, the conclusion follows.
Now, let $p\in [1,+\infty)$ be fixed. Consider the function 
\[ J:f\mapsto \frac1p\int_{-1}^1 u_{f}^p\] where, since $\O=\O^\sharp$ is fixed, we dropped the subscript and used $u_{f,\O}=u_f$. The map $J$ is G\^ateaux differentiable by elliptic estimates, and one obtains, using standard methods, that, for any $h\in L^\infty((-1;1))$, for any $f\in L^\infty((-1;1))$, the derivative
\[ J'(f)[h]:=\lim_{t\to 0^+}\frac{J(f+th)-J(f)}t\] is given by 
\[ J'(f)[h]=\int_{-1}^1 \omega_f h\text{ with }\begin{cases}-\Delta \omega_f=u_f^{p-1}&\text{ in }(-1;1),\, \\ \omega_f(\pm1)=0.\end{cases}\] Now, by convexity of $J$, we know that 
\[ J(f_{\gamma,\delta})-J(f_\gamma)\geq J'(f_\gamma)[f_{\gamma,\delta}-f_\gamma]=\int_{-1}^1 \omega_{f_\gamma}(f_{\gamma,\delta}-f_\gamma).\]
Equivalently, 
\begin{equation}\label{Eq:Mouais}J(f_\gamma)-J(f_{\gamma,\delta})\leq \int_{-1}^1 (f_\gamma-f_{\gamma,\delta})\omega_{f_\gamma}.\end{equation}
Note that $\omega_{f_\gamma}$ is radially symmetric decreasing. By elliptic regularity, $\omega_{f_\gamma}\in \mathscr C^{2,\eta}((-1;1))$ for any $\eta\in(0;1)$ and $-\frac12\omega_{f_\gamma}''(0)=c_0>0$.  In particular, we may approximate, for $\delta>0$ small enough, 
 \[ \omega_{f_\gamma}(x)-\omega_{f_\gamma}(0)=-c_0|x|^2+\underset{x\to 0}O(|x|^{2+\eta}).\]
Now, note that by equimeasurability
\begin{align*}
\int_{-1}^1 \omega_{f_\gamma}(f_\gamma-f_{\gamma,\delta})&=\int_{-1}^1 (\omega_{f_\gamma}-\omega_{f_\gamma}(0))(f_\gamma-f_{\gamma,\delta}).
\end{align*}
However, on the one hand, we have
\begin{align*}
\int_0^1 x^2 (f_\gamma-f_{\gamma,\delta}) dx&=\int_0^\delta x^2((x+\delta)^\gamma-x^\gamma)dx+\int_{\delta}^{2\delta} x^2((x-\delta)^\gamma-x^\gamma)dx
\\&=\delta^{3+\gamma}\int_0^1 y^2((1+y)^\gamma-y^\gamma)dy+\delta^{3+\gamma}\int_0^1 (1+y)^2(y^\gamma-(1+y)^\gamma)dy
\\&=-\delta^{3+\gamma}\int_0^1 (1+2y)\left((1+y)^\gamma-y^\gamma\right)dy
=-C' \delta^{3+\gamma}
\end{align*}
for some $C'>0$. On the other hand, by the same changes of variables,
\begin{align*}
\int_0^1 |x|^{2+\eta}\cdot|f_{\gamma,\delta}-f_\gamma|(x)dx&=\underset{\delta\to 0}O\left(\delta^{3+\eta+\gamma}\right).
\end{align*}
In particular,
\[ \Vert u_{f_{\gamma}}\Vert_{L^p((-1,1))}^p-\Vert u_{f_{\gamma,\delta}}\Vert_{L^p((-1,1))}^p=\underset{\delta\to 0^+}O\left( \delta^{3+\gamma}\right).\]Now,  let $\e>0$ and  $\gamma\in (0;1)$ be such that $\frac{3+\gamma}{1+\gamma}\in (3-\e;3)$. We obtain for some constant $C$ that does not depend on $\delta$
    \[\frac{\Vert u_{f_{\gamma}}\Vert_{L^p((-1,1))}^p-\Vert u_{f_{\gamma,\delta}}\Vert_{L^p((-1,1))}^p}{\Vert f_{\gamma,\delta}-f_\gamma\Vert_{L^1((-1,1))}^{3-\e}}\leq  C \delta^{(3+\gamma)-(3-\e)(1+\gamma)}\underset{\delta\to 0^+}\to0, \] and the conclusion follows.


\item\underline{No quadratic $L^1$ stability in dimension $d\in \{2,3\}$ if $f$ is not bounded: } Let $d=2$ or 3 and let, for any $\gamma>0$, 
 \[ g_\gamma:x\mapsto \frac1{| x|^\gamma}.\] In order to ensure that $g_\gamma\in L^2(\B(0;1))$, we require
 \[ \int_0^1 r^{d-2\gamma-1}dr<+\infty,\] or, equivalently, 
 \[\gamma<\frac{d}2.\]
 We define, for small $\delta>0$, the competitor:
 \[ g_{\gamma,\delta}:x\mapsto \begin{cases}
 (|x|^d+\delta^d)^{-\frac{\gamma}d}\text{ if } |x|\in(0;\delta),
 \\ (|x|^d-\delta^d)^{-\frac{\gamma}d}\text{ if }|x|\in (\delta;2^{\frac1d}\delta)\\ g_\gamma(x)\text{ else.}
 \end{cases}\] $g_{\gamma,\delta}$ and $g_\gamma$ are equimeasurable by construction. Notice now the following (the computations  are similar to the ones for $f_\gamma$ and $f_{\gamma,\delta}$ see \eqref{Eq:fg}-\eqref{Eq:fgd}): 
 \begin{align*}
        \inf_{z\in \R^d}\Vert g_{\gamma,\delta}-g_\gamma(\cdot+z)\Vert_{L^1}&=\int_{\B(0;1)}|g_{\gamma,\delta}-g_\gamma|\\&=d\omega_d\Bigg[
\int_0^\delta r^{d-1}
\left(r^{-\gamma}-(r^d+\delta^d)^{-\gamma/d}\right)\,dr\\
+&\int_\delta^{2^{1/d}\delta}r^{d-1}
\left((r^d-\delta^d)^{-\gamma/d}-r^{-\gamma}\right)\,dr
\Bigg]
        \\&=C_{\gamma}\cdot \delta^{d-\gamma}
 \end{align*} for some constant $C_\gamma$ independent of $\delta$. Similarly, for any $\beta>\gamma-d$, we have 
 \begin{equation}\label{Eq:Roo}
     \int_{\B(0;1)}|x|^\beta(g_\gamma(x)-g_{\gamma,\delta}(x))dx=C_{\beta,\gamma}\cdot\delta^{d-\gamma+\beta}
 \end{equation}
 and 
  \begin{equation}\label{Eq:Roo2}
     \int_{\B(0;1)}|x|^\beta|g_\gamma(x)-g_{\gamma,\delta}(x)|dx=C_{\beta,\gamma}'\cdot\delta^{d-\gamma+\beta}
 \end{equation}
 Now, let $p$ be fixed and set 
 \[ J(f):=\frac1p\int_{\B(0;1)}u_f^p.\] By the same convexity arguments that yielded \eqref{Eq:Mouais} we obtain 
 \[ J(g_\gamma)-J(g_{\gamma,\delta})\leq \int_{\B(0;1)} \omega_{g_\gamma} \left(g_{\gamma}-g_{\gamma,\delta}\right)\] where
 \[\begin{cases}
     -\Delta \omega_{g_\gamma}=u_{g_\gamma}^{p-1}&\text{ in }\B(0;1),\, 
     \\ \omega_{g_\gamma}=0&\text{ on }\partial \B(0;1).
 \end{cases}\]Note that $\omega_{g_\gamma}$ is radially symmetric. Furthermore, as $g_\gamma\in L^2$, if $d<4$, $u_{g_\gamma}\in \mathscr C^{0,\eta}$ for some $\eta>0$ by the Sobolev embedding $W^{2,2}\hookrightarrow \mathscr C^{0,\eta}$. This entails $\omega_{g_\gamma}\in \mathscr C^{2,\eta}(\B(0;1))$ for some $\eta>0$. In particular, we may approximate, for $\delta>0$ small enough, 
 \[ \omega_{g_\gamma}(x)-\omega_{g_\gamma}(0)=-c_0|x|^2+\underset{x\to 0}O(|x|^{2+\eta})\] where $c_0>0$ as $0$ is a non-degenerate maximiser of $\omega_{g_\gamma}$. By equimeasurability, we deduce that 
 \[\int_{\B(0;1)} \omega_{{g_\gamma}}\cdot(g_{\gamma,\delta}-g_\gamma)=-c_0\int_{\B(0;1)}|x|^2(g_{\gamma,\delta}-g_\gamma)+\int_{\B(0;1)}\underset{x\to 0}O(|x|^{2+\eta})(g_{\gamma,\delta}-g_\gamma)\underset{\delta\to 0^+}\sim -C \delta^{d+2-\gamma} \]where we used \eqref{Eq:Roo}. 
 
 In particular, 
 \[ \frac{\Vert u_{g_\gamma}\Vert_{L^p(\B(0;1))}^p-\Vert u_{g_{\gamma,\delta}}\Vert_{L^p(\B(0;1))}^p}{\inf_{z\in \R^d}\Vert g_{\gamma,\delta}-g_\gamma(\cdot+z)\Vert_{L^1(\R^d)}^2}=\underset{\delta\to 0^+}O\left(\delta^{2+\gamma-d}\right)\] for some $C>0$. Taking $\gamma\in \left(d-2;\frac{d}2\right)$, which is always possible for $d\in \{2,3\}$, we reach the conclusion.
\end{enumerate}
\printbibliography

@misc{arXiv:2609.07772,
 author = {Andr{\'e} Guerra and Jo{\~a}o Miguel Machado and Jo{\~a}o P. G. Ramos},
 title = {On the effective sharp stability of the {Faber}-{Krahn} inequality},
 year = {2026},
 howpublished = {Preprint, {arXiv}:2609.07772 [math.{AP}] (2026)},
 url = {https://arxiv.org/abs/2609.07772},
 arXiv = {arXiv:2609.07772}
}

@misc{zbMATH03073200,
 author = {Hardy, G. H. and Littlewood, J. E. and P{\'o}lya, George},
 title = {Inequalities. 2nd ed.},
 year = {1952},
 language = {English},
 howpublished = {Cambridge, {Engl}.: {At} the {University} {Press}. {XII}, 324 p. (1952).},
 zbMATH = {3073200},
 Zbl = {0047.05302}
}

@article{K21,
	author = {Kim, Daesung},
	doi = {10.1307/mmj/1593136867},
	fjournal = {Michigan Mathematical Journal},
	issn = {0026-2285},
	journal = {Mich. Math. J.},
	language = {English},
	number = {3},
	pages = {615--634},
	title = {Quantitative inequalities for the expected lifetime of {Brownian} motion},
	volume = {70},
	year = {2021},
	zbl = {1481.60162},
	zbmath = {7418473}}

@article{Lemou,
	author = {Lemou, Mohammed},
	journal = {Communications in Mathematical Physics},
	pages = {695--727},
	publisher = {Springer},
	title = {Extended rearrangement inequalities and applications to some quantitative stability results},
	volume = {348},
	year = {2016}}

@article{Burton1989,
	author = {Burton, G. R.},
	issn = {1873-1430},
	journal = {Annales de l'I.H.P. Analyse non lin{\'e}aire},
	language = {en},
	number = {4},
	pages = {295--319},
	series = {Annales de l'{I}.{H}.{P}. {Analyse} non lin{\'e}aire},
	title = {Variational problems on classes of rearrangements and multiple configurations for steady vortices},
	url = {http://www.numdam.org/item/?id=AIHPC_1989__6_4_295_0},
	urldate = {2025-01-28},
	volume = {6},
	year = {1989}}

@article{zbMATH07439881,
	author = {Mazari, Idriss},
	doi = {10.1016/j.na.2021.112649},
	fjournal = {Nonlinear Analysis. Theory, Methods \& Applications. Series A: Theory and Methods},
	issn = {0362-546X},
	journal = {Nonlinear Anal., Theory Methods Appl., Ser. A, Theory Methods},
	language = {English},
	note = {Id/No 112649},
	pages = {48},
	title = {Quantitative estimates for parabolic optimal control problems under {{\(L^\infty\)}} and {{\(L^1\)}} constraints in the ball: quantifying parabolic isoperimetric inequalities},
	volume = {215},
	year = {2022},
	zbl = {1478.49003},
	zbmath = {7439881}}

@article{zbMATH05578735,
	author = {Fusco, Nicola and Maggi, Francesco and Pratelli, Aldo},
	doi = {10.4007/annals.2008.168.941},
	fjournal = {Annals of Mathematics. Second Series},
	issn = {0003-486X},
	journal = {Ann. Math. (2)},
	language = {English},
	number = {3},
	pages = {941--980},
	title = {The sharp quantitative isoperimetric inequality},
	url = {annals.math.princeton.edu/annals/2008/168-3/p06.xhtml},
	volume = {168},
	year = {2008},
	zbl = {1187.52009},
	zbmath = {5578735}}

@article{zbMATH03962241,
	author = {Alvino, A. and Lions, Pierre-Louis and Trombetti, G.},
	doi = {10.1017/S0308210500014475},
	fjournal = {Proceedings of the Royal Society of Edinburgh. Section A. Mathematics},
	issn = {0308-2105},
	journal = {Proc. R. Soc. Edinb., Sect. A, Math.},
	language = {English},
	pages = {37--48},
	title = {A remark on comparison results via symmetrization},
	volume = {102},
	year = {1986},
	zbl = {0597.35005},
	zbmath = {3962241}}

@article{zbMATH03531830,
	author = {Talenti, Giorgio},
	fjournal = {Annali della Scuola Normale Superiore di Pisa. Classe di Scienze. Serie IV},
	issn = {0391-173X},
	journal = {Ann. Sc. Norm. Super. Pisa, Cl. Sci., IV. Ser.},
	language = {English},
	pages = {697--718},
	title = {Elliptic equations and rearrangements},
	url = {https://eudml.org/doc/83735},
	volume = {3},
	year = {1976},
	zbl = {0341.35031},
	zbmath = {3531830}}

@book{zbMATH07035366,
	author = {Baernstein, Albert II and Drasin, David and Laugesen, Richard S.},
	doi = {10.1017/9781139020244},
	fseries = {New Mathematical Monographs},
	isbn = {978-0-521-83047-8; 978-1-139-02024-4},
	language = {English},
	publisher = {Cambridge: Cambridge University Press},
	series = {New Math. Monogr.},
	title = {Symmetrization in analysis. {With} a foreword by {Walter} {Hayman}},
	volume = {36},
	year = {2019},
	zbl = {1509.32001},
	zbmath = {7035366}}

@article{zbMATH08185219,
	author = {Acampora, P. and Lamboley, J.},
	doi = {10.1007/s00205-026-02179-3},
	fjournal = {Archive for Rational Mechanics and Analysis},
	issn = {0003-9527},
	journal = {Arch. Ration. Mech. Anal.},
	language = {English},
	note = {Id/No 24},
	number = {3},
	pages = {54},
	title = {Sharp quantitative {Talenti} inequality in particular cases},
	volume = {250},
	year = {2026},
	zbl = {1587.35114},
	zbmath = {8185219}}

@article{zbMATH06898789,
	author = {Feldman, William M.},
	doi = {10.1137/17M1143009},
	fjournal = {SIAM Journal on Mathematical Analysis},
	issn = {0036-1410},
	journal = {SIAM J. Math. Anal.},
	language = {English},
	number = {3},
	pages = {3303--3326},
	title = {Stability of {Serrin}'s problem and dynamic stability of a model for contact angle motion},
	volume = {50},
	year = {2018},
	zbl = {1394.35295},
	zbmath = {6898789}}

@article{zbMATH06464861,
	author = {Brasco, Lorenzo and De Philippis, Guido and Velichkov, Bozhidar},
	doi = {10.1215/00127094-3120167},
	fjournal = {Duke Mathematical Journal},
	issn = {0012-7094},
	journal = {Duke Math. J.},
	language = {English},
	number = {9},
	pages = {1777--1831},
	title = {Faber-{Krahn} inequalities in sharp quantitative form},
	volume = {164},
	year = {2015},
	zbl = {1334.49149},
	zbmath = {6464861}}

@article{arXiv:2311.18617,
  title = {The Talenti comparison result in a quantitative form},
  ISSN = {0391-173X},
  url = {http://dx.doi.org/10.2422/2036-2145.202404_010},
  DOI = {10.2422/2036-2145.202404_010},
  journal = {Annali Scuola Normale Superiore-- Classe di Scienze},
  publisher = {Scuola Normale Superiore - Edizioni della Normale},
  author = {Amato,  Vincenzo and Barbato,  Rosa and Masiello,  Alba Lia and Paoli,  Gloria},
  year = {2024},
  month = Dec,
  pages = {31}
}

@article{CianFer08,
	author = {Cianchi, Andrea and Ferone, Adele},
	doi = {10.1112/jlms/jdm116},
	fjournal = {Journal of the London Mathematical Society. Second Series},
	issn = {0024-6107},
	journal = {J. Lond. Math. Soc., II. Ser.},
	language = {English},
	number = {3},
	pages = {581--592},
	title = {A strengthened version of the {Hardy}-{Littlewood} inequality},
	volume = {77},
	year = {2008},
	zbl = {1155.26017},
	zbmath = {5286376}}
\end{document}